\documentclass[reqno]{amsart}

\usepackage{amsaddr}
\usepackage{amsfonts}
\usepackage{graphicx}
\usepackage{epstopdf}
\usepackage{algorithmic}
\usepackage[many]{tcolorbox}
\usepackage{dsfont}
\usepackage{hyperref}
\usepackage{cleveref}
\usepackage{float}

\newtheorem{theorem}{Theorem}[section]

\newtheorem{lemma}[theorem]{Lemma}

\theoremstyle{definition}
\newtheorem{definition}[theorem]{Definition}
\newtheorem{remark}{Remark}

\newcommand{\e}{\varepsilon}
\newcommand{\Z}{\mathbb{Z}}
\newcommand{\R}{\mathbb{R}}

\newcommand{\dive}{\operatorname{div}}

\newcommand{\di}[1]{\,\mathrm{d}#1}

\newtcolorbox{mybox}[1]{%
    tikznode boxed title,
    enhanced,
    arc=0mm,
    interior style={white},
    attach boxed title to top left= {yshift=-\tcboxedtitleheight/2-0.05cm, xshift=0.7cm},
    fonttitle=\small\bfseries,
    colbacktitle=white,coltitle=black,
    boxed title style={size=small,colframe=white,boxrule=0pt},
    title={#1}}

\title[Homogenization of a Model for Fibre-Reinforced Hydrogels] 
      {Homogenization of a Poroelasticity Model for Fibre-Reinforced Hydrogels}

\author{Michael Eden}
\address{University Bremen, Germany}
\email{eden.michael@uni-bremen.de}
\author{Hari Shankar Mahato}
\address{IIT Kharagpur, India}
 \email{hsmahato@maths.iitkgp.ac.in}

\begin{document}

\maketitle
%
%
%
%

\begin{abstract}
In this paper, the analysis and homogenization of a poroelastic model for the hydro-mechanical response of fibre-reinforced hydrogels is considered.
Here, the medium in question is considered to be a highly heterogeneous two-component media composed of a connected fibre-scaffold with periodically distributed inclusions of hydrogel.
While the fibres are assumed to be elastic, the hydromechanical response of hydrogel is modeled via \emph{Biot's poroelasticity}.

We show that the resulting mathematical problem admits a unique weak solution and investigate the limit behavior (in the sense of two-scale convergence) of the solutions with respect to a scale parameter, $\e$, characterizing the heterogeneity of the medium.
While doing $\e\to 0$, we arrive at an effective model where the micro variations of the pore pressure give rise to a micro stress correction at the macro scale.
\end{abstract}

\section{Introduction}
Fibre-reinforced hydrogels (FIHs) are composites of a synthetic hydrogel\footnote{A hydrogel is a network of hydrophilic polymer chains; think \emph{edible jelly} for an every day life example.} reinforced by a scaffold of microfibres, see \Cref{fig:1}.
They are used in tissue engineering (e.g., cartilage, tendon and ligament tissue, and vascular tissue \cite{PW17}) where the FIH is used as a surrogate framework for \emph{in vitro} growth.
We refer to \cite{CH18,MHL10,PW17} for the biochemical details regarding this process, but in short: host stem cells are seeded in the FIH where they are able to grow in an lab environment into fully functional tissue which can then be transplanted back into the host.
With the highly hydrated polymer network of hydrogels, FIHs mimic the environment of \emph{natural extracellular matrices} (ECMs), while the fibre scaffold improves the mechanical properties, see \cite{CH18} and references therein.
Without this reinforcement, it is very difficult to get mechanical strength and structural resilience comparable to its native biological counterpart \cite{C09,PW17}.

\begin{figure}[h]
\centering
\includegraphics[width=.6\textwidth]{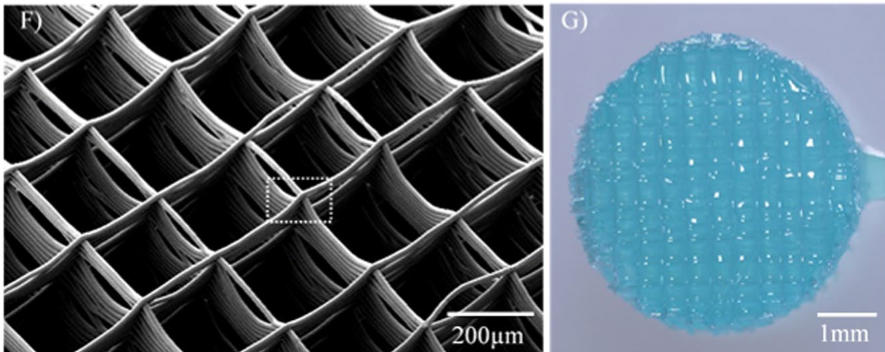}%
\caption{The periodic fibre scaffold empty (left) and saturated with hydrogel (right). This figure is taken from \cite{CH18} under a Creative Commons license (see: \url{http://creativecommons.org/license/by/4.0/})}%
\label{fig:1}%
\end{figure}

This kind of in vitro tissue engineering is a relatively new approach and has some important advantages when compared to alternative treatments: it does not involve donor cells, which removes the danger of adverse immune response, and it has the prospect of enabling therapies that are "cost-effective, time-efficient, and single procedure"(\cite{PW17}).

In practice, the filament spacing of the scaffold is usually in the range of $\mu m$ while the overall size of an FIH is in the range of $mm$ to $cm$ (\Cref{fig:1}).
Due to this scale heterogeneity, the effective hydromechanical properties of FIHs are not yet fully understood and, as a consequence, there is an interest in describing, modeling, and calculating their effective properties based on the underlying microstructure, cf.~\cite{CH18,CK18}.

In this work, we present a rigorously derived effective model for the hydromechanical properties of an FIHs based on a microstructure model describing the interplay between hydrogel and fibre structure.
This micromodel assumes the fibre scaffold to be elastic and the hydrogel to be linearly poroelastic (\emph{Biot's poroelasticity}).
After showing that this micromodel has a unique solution (well-posedness), a limit process in the context of mathematical homogenization is employed to arrive at the effective model.
In particular, this method also gives us effective material parameters like the elastic modulus.
To our knowledge, this is first mathematical work rigorously treating this particular problem.

In \cite{CH18}, a two-scale finite element computational model is proposed, where the hydrogel is assumed to behave like a \emph{Neo Hookean solid} (hyperelastic).
More closely related to our work, in \cite{CK18}, an effective model is derived from a two-phase elastic-poroelastic microproblem via formal asymptotic expansion.

It is worth noting that similar models to the micromodel considered in this work with some of the same features emerge in different applications as well.
In \cite{SM02}, a general mathematical analysis of \emph{Biot's poroelascticity} model is presented.
In ensuing work, mathematical homogenization scenarios in the context of double poroelasticity were explored, see, e.g., \cite{A11,A13,EB14}.
For further examples in the context of the related thermoelasticity, we refer to \cite{EKK02,EM17,ETT15}.

Regarding the structure of the article, in \Cref{sec:2}, we introduce the setting and the model as well as present the main results.
The \Cref{sec:3,sec:4} are dedicated to the study of the microproblem and the proof of the homogenization result, respectively.
\section{Setting of the problem and main results}\label{sec:2}
In this section, we provide the detailed geometric setup as well as the mathematical model that we are considering. 
In addition, we present our main results, namely \Cref{thm_existence,thm_homogenization}, which are proved in the ensuing sections.

In the following, let $\Omega\subset\R^n$ be a bounded Lipschitz domain representing the overall system and let $S=(0,T)$, $T>0$, represent the time interval of interest.
In addition, we denote the outer normal vector of $\Omega$ with $\nu=\nu(x)$.
Let $Y=(0,1)^3$ be the open unit cell in $\R^3$.
Take $Y^{f},$ $Y^{g}\subset Y$  two disjoint open sets, such that $Y^{f}$ is connected, such that $\Gamma:=\overline{Y^{f}}\cap\overline{Y^{g}}=\partial Y^{g}$, $\overline{Y^{g}}\subset Y$, and $Y=Y^{f}\cup Y^{g}\cup \Gamma$, see \Cref{f:1}.
With $n_\Gamma=n_\Gamma(y)$, $y\in\Gamma$, we denote the normal vector of $\Gamma$ pointing outwards of $Y^{g}$.

For $\e>0$, we introduce the $\e Y$-periodic initial domains $\Omega_\e^{f}$, $\Omega_\e^{g}$ and the interface $\Gamma_\e$ representing the fibre domain, gel domain and the boundary between fibre and gel, respectively. By $\partial \Omega^f_\e$ we denote the outer boundary of $\Omega$. Via ($i=f,g$)
\begin{align*}
	\Omega_\e^{i}=\Omega\cap\left(\bigcup_{k\in\Z^3}\e(Y^{i}+k)\right),\qquad
	\Gamma_\e=\Omega\cap\left(\bigcup_{k\in\Z^3}\e(\Gamma+k)\right).
\end{align*}

In the following, $\chi_\e^{i}\colon\Omega\to\{0,1\}$ ($i=f,g$) denotes the characteristic functions corresponding to $\Omega_\e^i$.

\begin{remark}
Please note that by this design $\Omega_\e^{f}$ is connected and $\Omega_\e^g$ is disconnected, which does not perfectly match the scaffold depicted in \Cref{fig:1} where holes between cells are clearly visible.
As in the similar work done in \cite{CK18}, however, we assume the fibre scaffold to be closed. In particular, two separate neighboring gel cells can only interact via the fibre scaffold.
\end{remark}
\begin{remark}
Furthermore, in reality, the fibre scaffold is very thin in relation to the size of an individual cell, see \Cref{fig:1}.
As a consequence, it might make sense to introduce an additional scale parameter measuring this thickness.
\end{remark}

\begin{figure}[!ht]\centering
	\begin{tikzpicture}[scale=0.6]
		\begin{scope}[rotate=0]
			\draw[fill=black!15,very thick] (5,0) rectangle (9,4);
			\draw[fill=black!15] (5,0) rectangle (9,4);
				\draw (8.85,2) -- (9.15,2);
				\draw (8.85,3) -- (9.15,3);
				\draw(9,2.5) node[right] {$\e$};
				\foreach \y in {1,2,3,4}
				\foreach \x in {1,2,3,4}
					\draw[fill=black!50] (4.5+\x,-0.5+\y) circle (0.44cm);
			\draw(6,3.8) -- (5.9,4.4);
					\draw (5.7,4.4) node[above] {$\Omega_f^\e$};	
			\draw (6.8,3.8) -- (6.8,4.4);
					\draw (7,4.4) node[above] {$\Gamma^\e$};
			\draw (8.5,3.5) -- (8.2,4.4);	
					\draw (8,4.4) node[above] {$\Omega_g^\e$};
			\draw[fill=black!15,very thick] (0,0) rectangle (4,4);
			\draw[fill=black!15] (0,0) rectangle (4,4);
			\draw[fill=black!50] (2,2) circle (1.7cm);
			\draw (1.5,3.2) -- (1.2,4.4);	
			\draw (1,4.4) node[above] {$Y_g$};
			\draw (3,3.7) -- (3.3,4.4);	
			\draw (3.15,4.4) node[above] {$Y_f$};
			\end{scope}
			\end{tikzpicture}
			\caption{The periodicity cell $Y$ and the geometrical setup of the periodic domain}
			\label{f:1}
	\end{figure}
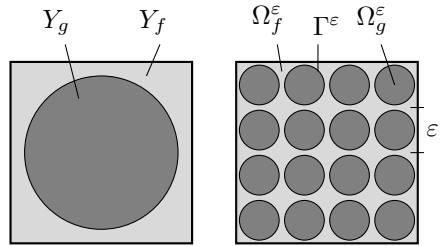

Now, let $w_\e\colon\Omega_\e^f\to\R^3$ represent the deformation in the fibre part.
Assuming that the mechanical response of the fibre scaffold is governed by quasi-stationary elasticity, we then have
\begin{subequations}
\begin{align}
-\nabla\cdot(\mathcal{C}e(w_\e))=f\quad\text{in}\quad S\times\Omega_\e^f.\label{eq:2.1a}
\end{align}

Here, $e(w_\e)=1/2(\nabla w_\e+\nabla w_\e^T)$ denotes the linearized strain tensor, $\mathcal{C}$ the elasticity tensor, and $f$ possible volume forces.
The hydrogel is itself a composite (polymer chains saturated with water) with complex mechanical properties.
In this work, as in many others, e.g., \cite{CYB18,CK18,KMh12,LN17}, we model it as a linear poroelastic material.
To that end, let $u_\e\colon\Omega_\e^g\to\R^3$ denote the solid deformation in the gel part and let $p_\e\colon\Omega_\e^g\to\R$ denote the pore pressure.
The model of Biot's linear poroelasticity is given by
\begin{alignat}{2}
-\nabla\cdot\left(\mathcal{D}e(u_\e)-\alpha p_\e\mathds{I}\right)&=g&\quad &\text{in}\ \ S\times\Omega_\e^g,\label{eq:2.1b}\\
\partial_t\left(cp_\e+\alpha\dive u_\e\right)-\nabla\cdot\left(\e^2 K\nabla p_\e\right)&=h&\quad &\text{in}\ \ S\times\Omega_\e^g.\label{eq:2.1c}
\end{alignat}

Here, $\mathcal{D}$ is the elasticity tensor, $c$ is the \emph{Biot modulus}, $\alpha$ is the \emph{Biot-Willis} parameter, and $\e^2K$ is the permeability.
$\mathds{I}$ denotes the unit matrix, and, again, $g$ and $h$ are possible volume sources.
Please note that this particular $\e^2$-scaling for the permeability in the gel part is the typical choice for these kinds of two-scale problems, cf.~\cite{CS99,EB14,FAZM11,Y09}.

Regarding the interaction of the two different phases, we assume both deformations and forces to be continuous across the interface $\Gamma_\e$.
That is,
\begin{alignat}{2}
\left(\mathcal{D}e(u_\e)-\alpha p_\e\mathds{I}\right)n_\e&=(\mathcal{C}e(w_\e))n_\e&\quad &\text{on}\quad S\times\Gamma_\e,\\
u_\e&=w_\e&\quad &\text{on}\quad S\times\Gamma_\e,\label{seq:5}
\end{alignat}

In addition, we assume no gel flux between the different cells
\begin{alignat}{2}
-\e^2 K\nabla p_\e\cdot n_\e&=0&\quad &\text{on}\quad S\times\Gamma_\e.
\end{alignat}

Finally, we close the model with homogeneous outer boundary and inital conditions:
\begin{alignat}{2}
u_\e&=0&\quad &\text{on}\quad S\times\partial\Omega_\e^f,\\
p_\e(0)&=0&\quad &\text{in}\quad \Omega_\e^g\label{seq:8}.
\end{alignat}
\end{subequations}

We assume all scalar coefficients to be constant and positive and all matrix and tensor coefficients to be constant, symmetric, and positive definite.

\begin{remark}
It is straightforward to extend our results to different boundary and initial conditions as well as to certain \emph{well behaved} non-homogeneous coefficients.
\end{remark}


%

We note that, due to interface condition \eqref{seq:5}, we can expect $U_\e\in H_0^1(\Omega)$ where

$$
U_\e:=
\begin{cases} w_\e\quad& \text{in}\ \ S\times\Omega_\e^f\\
u_\e\quad& \text{in}\ \ S\times\Omega_\e^g.
\end{cases}
$$
In the following, we will denote the zero extension of any function $\psi$ defined on $\Omega_\e^f$ or $\Omega_\e^g$ to the whole of $\Omega$ by $\widehat{\psi}$; with that we have $U_\e=\widehat{w_\e}+\widehat{u_\e}$.

Setting $\mathcal{A}=\chi_{\Omega_\e^f}\mathcal{C}+\chi_{\Omega_\e^g}\mathcal{D}$, a corresponding weak form of the full model is then given by:

\begin{mybox}{Weak form}
Find $(U_\e,p_\e)\in L^2(S;H_0^1(\Omega)\times H^1(\Omega_\e^g))$ such that $\partial_t\left(cp_\e+\alpha\nabla\cdot U_\e\right)\in L^2(S;H^1(\Omega_\e^g)^*)$ and 

\begin{subequations}
\begin{equation}\label{weak.1}
\int_\Omega\mathcal{A}\,e(U_\e)\colon e(v)\di{x}
-\int_{\Omega_\e^g}\alpha p_\e\nabla\cdot v\di{x}
=\int_{\Omega_\e^f}f\cdot v\di{x}+\int_{\Omega_\e^g}g\cdot v\di{x},
\end{equation}
\begin{align}\label{weak.2}
\left\langle\partial_t\left(cp_\e+\alpha\nabla\cdot U_\e\right),\varphi\right\rangle_{H^1(\Omega_\e^g)^*}
+\int_{\Omega_\e^g}\e^2 K\nabla p_\e\cdot\nabla\varphi\di{x}
&=\int_{\Omega_\e^g}h\varphi\di{x}
\end{align}
for all $(v,\phi)\in H_0^1(\Omega)\times H^1(\Omega_\e^g)$ and for almost all $t\in S$.
\end{subequations}
\end{mybox}

Here, and in the following, we assume that $f\in L^2(\Omega_\e^f)^3$, $g\in L^2(\Omega_\e^g)^3$, and $h\in L^2(\Omega_\e^g)$.
Also, for any Banach space $V$, $V^*$ denotes its topological dual and the bracket $\langle \cdot,\cdot\rangle$ indicates the corresponding dual pairing.

\subsection{Main results} In the following, we present our main results.
Namely, the existence result and $\e$-controlled estimates for the $\e$-problem, see \Cref{thm_existence}, and the final homogenization result, see \Cref{thm_homogenization}. 
The detailed proofs of these results can be found in Chapters 3 and 4, respectively.

\begin{theorem}\label{thm_existence}
There is a unique solution $(U_\e,p_\e)\in L^2(S;H_0^1(\Omega)\times H^1(\Omega_\e^g))$ with $\partial_t\left(cp_\e+\alpha\nabla\cdot U_\e\right)\in L^2(S;H^1(\Omega_\e^g)^*)$ satisfying equations \cref{weak.1,weak.2} for all test functions $(v,\phi)\in H_0^1(\Omega)\times H^1(\Omega_\e^g)$ and for almost all $t\in S$.
In addition,
 
$$
\sup_{\e>0}\left(\|p_\e\|_{L^\infty(S;L^2(\Omega_\e^g))}^2
+\|U_\e(t)\|^2_{L^\infty(S;H^1_0(\Omega))}
+\e^2\|\nabla p_\e\|^2_{L^2(S\times\Omega_\e^g)}\right)<\infty.
$$ 
\end{theorem}
\begin{proof}
This theorem is proved in Chapter 3.
For the existence of a unique solution, see \Cref{3:lem_solution}. 
The estimates are provided in \Cref{3:lem_estimates}.
\end{proof}

\begin{theorem}[Homogenization result]\label{thm_homogenization}
There are functions $U\in L^2(S;H_0^1(\Omega))^3$ and $p\in L^2(S;L^2(\Omega ;H_{per}^{1}(Y)))$ such that $\widehat{U_\e}\to U$ in $L^2(S;H_0^1(\Omega))^3$ and $\widehat{p_\e}\rightharpoonup[p]_y$.
Moreover, the limit functions satisfy the following homogenized system:
\begin{subequations}
\begin{alignat}{2}
-\nabla\cdot\left(\mathcal{A}^he(U)+\mathcal{A}[e_y(\tilde{u})]_Y-\alpha[p]_{Y^g}\right)&=F&\quad&\text{in}\ \Omega,\label{ahres1}\\
U&=0&\quad&\text{on}\ \partial\Omega,\label{ahres2}\\
-\nabla_y\cdot\left(\mathcal{C}\, e_y(\tilde{u}_f)\right)&=0&\quad&\text{in}\ \ Y^f,\label{ahres3}\\
-\nabla_y\cdot\left(\mathcal{D}\, e_y(\tilde{u}_g)-\alpha\,p\mathds{I}_3\right)&=0&\quad&\text{in}\ \ Y^g,\label{ahres4}\\
-\left(\mathcal{D}\, e_y(\tilde{u}_g)-\alpha\,p\mathds{I}_3\right)n_\Gamma&=-\mathcal{C}\, e_y(\tilde{u}_f)n_\Gamma&\quad&\text{on}\ \ \Gamma,\label{ahres5}\\
\tilde{u}_g&=\tilde{u}_f&\quad&\text{on}\ \ \Gamma,\label{ahres6}\\
y&\mapsto \tilde{u}\quad &&Y\text{-periodic},\label{ahres7}\\
\partial_t\left(cp+\alpha^h:\nabla U+\dive_y\tilde{u}\right)-\nabla\cdot\left(\mathcal{K}\nabla_y p\right)
&=h&\quad&\text{in}\ \ S\times Y^f\label{ahres8}\\
-K_y\nabla p\cdot n_\Gamma&=0&\quad&\text{on}\ \ S\times\Gamma,\label{ahres9}\\
p(0)&=0&\quad&\text{on}\ \ \Gamma.\label{ahres0}
\end{alignat}
\end{subequations}

The additional micro deformations $\tilde{u}_f$ and $\tilde{u}_g$ ($\tilde{u}=\widehat{\tilde{u}_f}+\widehat{\tilde{u}_g}$) are a consequence of the micro variations of the pore pressure, see \cref{ahres3,ahres4,ahres5,ahres6,ahres7}.
Here, $[\psi]_A$ denotes integration of a function $\psi$ over a domain $A$.
\end{theorem}
\begin{proof}
This theorem is proved in Chapter 4. 
For the convergence results, see \Cref{lem:hom_limit}.
The homoegenized system is then deduced in starting with \cref{hom_1}.
For the definitions of the effective coefficients like $\mathcal{A}^h$, see \cref{sys:h1,sys:h3,sys:h4,sys:h5}
\end{proof}

This homogenized model exhibits several interesting features.
First and foremost, we have the additional micro deformations and stresses given via $\tilde{u}$ that arises via the micro variations of the pore pressure $p$ governed by \cref{ahres3,ahres4,ahres5,ahres6,ahres7}.

First, the micro pore pressure problem given by equations \cref{ahres8,ahres9,ahres0} is almost the standard micro system for the chosen geometrical setup and $\e$-scaling (see, e.g., \cite{A11,EB14} for similar homogenization results) where the isotropy of the pressure deformation coupling is lost ($\alpha\dive U_\e$ vs.~$\alpha^h:\nabla U)$.
The main difference is the additional micro dissipation term given by $\alpha\dive_y\tilde{u}$: due to the micro variations of the pore pressure and the resulting micro deformations and stresses, a purely macroscopic term like $\alpha^h:\nabla U$ is not sufficient to capture the mechanical dissipation.

For the same reasons, the macroscopical momentum problem (\cref{ahres1,ahres2}) includes an additional averaged micro stress contribution (namely $\mathcal{A}[e_y(\tilde{u})]_Y$) accounting for the stresses due to the deformations at the micro scale.
Those are governed by equations \cref{ahres3,ahres4,ahres5} and are solely a consequence of the micro variations of the pore pressures.

\section{Analysis of the $\e$-Problem}\label{sec:3}
In this section, we tend to the analysis of the weak form given by \cref{weak.1,weak.2}.
To this end, we will introduce an equivalent abstract linear operator formulation of the problem, see \cref{eq:operator_form}, and establish some important properties of the involved operators.

We start with the momentum balance equation and note that, for every $\psi\in H^{-1}(\Omega)^3$, there is a unique $U_\e\in H^1_0(\Omega)^3$ such that

\begin{align*}
\int_\Omega\mathcal{A}\, e(U_\e)\colon e(v)\di{x}
=\langle\psi,v\rangle_{H^{-1}(\Omega)^3}
\end{align*}
for all $v\in H_0^1(\Omega)^3$.
Since $\mathcal{A}$ is positive definite, this follows by the Lemma of Lax-Milgram by using Korns inequality.
Also, the induced operator $E_\e\colon H_0^1(\Omega)^3\to H^{-1}(\Omega)$ is a homeomorphism.
We set $H_\e= L^2(\Omega)^3\times L^2(\Gamma_\e)^3$ with the natural embedding $H_0^1(\Omega)^3\hookrightarrow H_\e$ given by $v\mapsto[v,v_{|\Gamma_\e}]$.
We introduce the $\e$-gradient operator

$$
\nabla_\e\colon H^1(\Omega_\e^g)\to H_\e \ \ \text{via}\ \ (\nabla_\e p,[f,g])_{H_\e}=\int_{\Omega_\e^g}\alpha \nabla p_\e f\di{x}-\int_{\Gamma_\e}\alpha p_\e n_\e\cdot g,
$$
as well as the $\e$-divergence operator $\nabla_\e\cdot$,\footnote{Here, we have identified the Hilbert space $H_\e$ with its dual. Moreover, $^*$ denotes the dual operator.}

$$
\nabla_\e\cdot:=-(\nabla_\e)^*\colon H_\e\to H^1(\Omega_\e^g)^*\ \ \text{via}\ \ 
\langle\nabla_\e\cdot[f,g], p\rangle_{H^1(\Omega_\e^g)}=-\left(\nabla_\e p,[f,g]\right)_{H_\e}.
$$
For smoother functions $v\in H^1_0(\Omega)^3$, we have via the divergence theorem

$$
\langle\nabla_\e\cdot v, p_\e\rangle_{H^1(\Omega_\e^g)}=-\int_{\Omega_\e^g}\alpha \nabla p_\e v\di{x}+\int_{\Gamma_\e}\alpha p_\e n_\e\cdot v=\int_{\Omega_\e^g}\alpha p_\e\nabla\cdot v
$$
which implies $\nabla_\e\cdot_{|_{H^1_0(\Omega)^3}}\colon H^1_0(\Omega)^3\to L^2(\Omega_\e^g)$.
Conversely, noting that

$$
\nabla_\e=-\left(\nabla_\e\cdot_{|_{H^1_0(\Omega)^3}}\right)^*_{|_{H^1(\Omega_\e^g)}}
$$
we can extend the $\e$-gradient operator to $L^2(\Omega_e^g)$-functions by setting

$$
\nabla_\e =-\left(\nabla_\e\cdot_{|_{H^1_0(\Omega)^3}}\right)^*.
$$
By abuse of notation, we will from now on only use $\nabla_\e$ and $\nabla_\e\cdot$, see also \Cref{f:diagram}.

{\renewcommand{\arraystretch}{1.2} 
	\begin{figure}[H]
	\centering 
		\large 
		\begin{tabular}{cccccc}
			$H_\e$	& $\stackrel{\nabla\cdot=-(\nabla^\e)^*}{\longrightarrow}$ 									&$ H^1(\Omega_g^\e)^*$&&\\ 
			$\bigcup$														&		&$\bigcup$																						&&\\
			$H_0^1(\Omega)^3$											& $\stackrel{\nabla\cdot}{\longrightarrow}$    &$L^2(\Omega_g^\e)$			&$\stackrel{\nabla^\e=-(\nabla^\e\cdot)^*}{\longrightarrow}$ &$H^{-1}(\Omega)^3$\\
			&&$\bigcup$				&	&$\bigcup$					\\
			&	&$ H^1(\Omega_g^\e)$					& $\stackrel{\nabla^\e}{\longrightarrow}$ 	&$H_\e$
		\end{tabular}
		\caption{Diagram of the interaction of the coupling operators, inspired by \cite{SM02}.}
		\label{f:diagram}
		\end{figure}}

We see that \cref{weak.1} is equivalent to

$$
E_\e U_\e+\nabla_\e p_\e=\mathcal{F}_\e\quad\text{in}\ H^{-1}(\Omega),
$$
where 

$$
\mathcal{F}_\e=\begin{cases}f\quad&\text{in}\ \Omega_\e^f\\ g\quad&\text{in}\ \Omega_\e^g\end{cases}.
$$
For any given $p_\e\in H^1(\Omega_g^\e)$, this leads to the solution

\begin{equation}\label{eq:u}
U_\e=-E_\e^{-1}\nabla_\e p_\e+E_\e^{-1}\mathcal{F}_\e\quad \text{in}\ H_0^1(\Omega).
\end{equation}
We go on introducing the linear operator

$$
\mathcal{K}_\e\colon H^1(\Omega_\e^{g})\to H^1(\Omega_\e^{g})^*
	\ \ \text{via}\ \  \langle\mathcal{K}_\e p,\varphi\rangle_{H^1(\Omega_\e^{g})^*}=\e^2(K\nabla p,\nabla\varphi)_{L^2(\Omega_\e^{g})}.
$$
For \cref{weak.2}, this leads to

\begin{equation}\label{eq:operator_form}
\partial_t\left(cp_\e-\nabla_\e\cdot E_\e^{-1}\nabla_\e p_\e\right)
+\mathcal{K}_\e p_\e
=\mathcal{H}_\e\quad\text{in}\ H^1(\Omega_\e^g)^*,
\end{equation}
where $\mathcal{H}_\e$ is given by

$$
\mathcal{H}_\e=h+\partial_t\nabla_\e\cdot E_\e^{-1}\mathcal{F}_\e.
$$

\begin{remark}
Please note that the operator $\nabla_\e\cdot E_\e^{-1}\nabla_\e$, although involving differential operators, is not itself a differential operator.
Formally, both $\nabla_\e$ and $\nabla_\e\cdot$ are differential operators of order 1 and $E_\e^{-1}$, being the inverse of an elliptic operator, will lift the function for two derivatives.
As a consequence, $\nabla_\e\cdot E_\e^{-1}\nabla_\e$ maps $L^2(\Omega_\e^g)$ into $L^2(\Omega_\e^g)$.
Also, for $\mathcal{F}_\e\in C^1(\overline{S};L^2(\Omega_\e^g))$, $\partial_t\nabla_\e\cdot E_\e^{-1}\mathcal{F}_\e$ is well defined as $\nabla_\e\cdot E_\e^{-1}$ is linear, bounded, and time independent.
\end{remark}

We set

$$
\mathcal{B}_\e:=c\mathrm{Id}-\nabla_\e\cdot E_\e^{-1}\nabla_\e\colon L^2(\Omega_\e^g)\to L^2(\Omega_\e^g)
$$
where $\mathrm{Id}$ denotes the idendity operator.

\begin{lemma}\label{b_operator}
The operator $\mathcal{B_\e}$ is linear, continuous, positive, and self-adjoint.
\end{lemma}
\begin{proof}
Linearity and continuity are clear since $\mathcal{B}_\e$ is composed of linear and continuous operators.
For the positivity, we observe ($p\in L^2(\Omega_\e^g)$)

\begin{align*}
\left(\mathcal{B}_\e p,p\right)_{L^2(\Omega_\e^g)}&=c\|p\|^2_{L^2(\Omega_\e^g)}-\left(\nabla_\e\cdot E_\e^{-1}\nabla_\e p,p\right)_{L^2(\Omega_\e^g)}\\
&=c\|p\|^2_{L^2(\Omega_\e^g)}+\langle\nabla_\e p,E_\e^{-1}\nabla_\e p\rangle_{H^1_0(\Omega)}\\
&\geq c\|p\|^2_{L^2(\Omega_\e^g)}+c_\mathcal{A}^{-1}\|\nabla_\e p\|_{H^{-1}(\Omega)^3}.
\end{align*}
Here, we have used the duality between the $\e$-gradient and $\e$-divergence operators.
The positivity of $E_\e^{-1}$ thus implies positivity of $\mathcal{B}_\e$.
Similarly, the self-adjointness follows directly from $E_\e^{-1}$ having this property.
\end{proof}

\begin{lemma}[Existence of a unique solution]\label{3:lem_solution}
Let $\mathcal{F}_\e\in C^1(\overline{S};L^2(\Omega))$ and $h\in L^2(S\times\Omega)$.
There is a unique $(U_\e,p_\e)\in L^2(S;H_0^1(\Omega)^3\times H^1(\Omega_\e^g))$ satisfying $p_\e(0)=0$ and $\partial_tp_\e\in L^2(S\times\Omega)$ solving the operator problem
\begin{alignat}{2}
E_\e U_\e+\nabla_\e p_\e&=\mathcal{F}_\e&\quad&\text{in}\ H^{-1}(\Omega),\\
\partial_t(\mathcal{B}_\e p_\e)+\mathcal{K}_\e p_\e&=\mathcal{H}_\e&\quad&\text{in}\ H^1(\Omega_\e^g)^*.\label{eq:2}
\end{alignat}
Moreover, we have $\partial_tU_\e\in L^2(S;H_0^1(\Omega)^3)$.
\end{lemma}

\begin{proof}
Since the operator $\mathcal{B}_\e$ is linear, continous, positive, and self-adjoint and $\mathcal{K}_\e$ is positive definite, there is a unique solution $p_\e\in L^2(S;H^1(\Omega_\e^g))$ of \cref{eq:2} such that $(\mathcal{B_\e}p_\e)(0)=0$ and $\partial_t(\mathcal{B}_\e p_\e)\in L^2(S;H^1(\Omega)^*)$, see \cite[Chapter III.3, Proposition 3.2]{S96}.
Since $\mathcal{B}_\e$ is injective, $\mathcal{B_\e}p_\e(0)=0$ implies $p_\e(0)=0$.
Moreover, we get the corresponding $U_\e\in L^2(S;H_0^1(\Omega)^3)$ via equation \cref{eq:u}.

At this point, it is not clear that time derivatives of $p_\e$ or $U_\e$ exist (we only have $\partial_t(\mathcal{B}_\e p_\e)\in L^2(S;H^1(\Omega)^*)$), but since $\mathcal{B}_\e$ is time-independent, linear, continuous, and strictly positive, we see that $\partial_t p_\e\in L^2(S;H^1(\Omega)^*)$.
As $\mathcal{H}_\e\in L^2(S\times\Omega)$, it follows that $\partial_t p_\e\in L^2(S\times\Omega)$ as well.
Now, we define $W_\e(t)\in H_0^1(\Omega)^3$ as

$$
W_\e(t)=-E_\e^{-1}\nabla_\e\partial_tp_\e+E_\e^{-1}\partial_t\mathcal{F}_\e
$$ 
which shows us that $\partial_tU_\e=W_\e\in L^2(S;H_0^1(\Omega)^3)$.
\end{proof}

With this result regarding the existence of a unique solution, we now turn to $\e$-controlled energy estimates.
Those are extremely important in the homogenization context as they will be used to facilitate the limit analysis $\e\to0$.
It is possible to arrive at slightly better estimate than given in the ensuing lemma, e.g., also boundedness of $\partial_tp_\e$ in $L^2(S\times\Omega)$, but we concentrate on the estimates needed for the homogenization.

\begin{lemma}[Estimates]\label{3:lem_estimates}
Every solution $(u_\e,p_\e)$ satisfies
\begin{equation}\label{eq:3.5}
\sup_{\e>0}\left(\|p_\e\|_{L^\infty(S;L^2(\Omega_\e^g))}^2
+\|U_\e(t)\|^2_{L^\infty(S;H^1_0(\Omega))}
+\e^2\|\nabla p_\e\|^2_{L^2(S\times\Omega_\e^g)}\right)<\infty.
\end{equation}
\end{lemma}
\begin{proof}
We formally test the weak formulation with $(\partial_t U_\e,p_\e)$\footnote{This is not rigorous as $\partial_t U_\e$ does not need to be in $H_0^1(\Omega)^3$; this gap can be closed, however, by using difference quotients and doing a limit analysis, see, e.g., \cite{GT01}.} and get
\begin{multline*}
\int_\Omega\left(\chi_{\Omega_\e^f}\mathcal{C}+\chi_{\Omega_\e^g}\mathcal{D}\right)e(U_\e)\colon e(\partial_tU_\e)\di{x}
-\int_{\Omega_\e^g}\alpha p_\e\nabla\cdot \partial_tU_\e\di{x}\\
=\int_{\Omega_\e^f}f\cdot \partial_tU_\e\di{x}+\int_{\Omega_\e^g}g\cdot \partial_tU_\e\di{x},
\end{multline*}
\begin{align*}
\left\langle\partial_t\left(cp_\e+\alpha\nabla\cdot u_\e\right),p_\e\right\rangle_{H^1(\Omega_\e^g)^*}
+\int_{\Omega_\e^g}\e^2 K\nabla p_\e\cdot\nabla p_\e\di{x}
&=\int_{\Omega_\e^g}hp_\e\di{x}.
\end{align*}
Since $K$ is positive definite (i.e., there is $c_K>0$ such that $Kx\cdot x\geq c_k|x|^2$ for all $x\in\R^3$), we infer that
\begin{align*}
\left\langle\partial_t\left(cp_\e+\alpha\nabla\cdot u_\e\right),p_\e\right\rangle_{H^1(\Omega_\e^g)^*}
+\e^2 c_K\|\nabla p_\e\|^2_{L^2(\Omega_\e^g)}
&\leq\int_{\Omega_\e^g}hp_\e\di{x}.
\end{align*}
Integrating both equations with respect to time for some $t>0$, we get
\begin{multline*}
\int_0^t\int_\Omega\left(\chi_{\Omega_\e^f}\mathcal{C}+\chi_{\Omega_\e^g}\mathcal{D}\right)e(U_\e)\colon e(\partial_tU_\e)\di{x}\di{\tau}
-\int_0^t\int_{\Omega_\e^g}\alpha p_\e\nabla\cdot \partial_tU_\e\di{x}\di{\tau}\\
=\int_0^t\int_{\Omega_\e^f}f\cdot \partial_tU_\e\di{x}\di{\tau}+\int_0^t\int_{\Omega_\e^g}g\cdot \partial_tU_\e\di{x}\di{\tau},
\end{multline*}
\begin{align*}
\int_0^t\left\langle\partial_t\left(cp_\e+\alpha\nabla\cdot u_\e\right),p_\e\right\rangle_{H^1(\Omega_\e^g)^*}\di{\tau}
+\int_0^t\e^2 c_K\|\nabla p_\e\|^2_{L^2(\Omega_\e^g)}\di{\tau}
&=\int_0^t\int_{\Omega_\e^g}hp_\e\di{x}\di{\tau}.
\end{align*}
Adding both the equations leads to
\begin{multline*}
c\int_0^t\langle\partial_tp_\e,p_\e\rangle_{H^1(\Omega_\e^g)^*}\di{\tau}\\
+\int_0^t\int_\Omega\left(\chi_{\Omega_\e^f}\mathcal{C}+\chi_{\Omega_\e^g}\mathcal{D}\right)e(U_\e)\colon e(\partial_tU_\e)\di{x}\di{\tau}
+\e^2 c_K\int_0^t\|\nabla p_\e\|^2_{L^2(\Omega_\e^g)}\di{\tau}\\
\leq\int_0^t\int_{\Omega_\e^g}hp_\e\di{x}\di{\tau}+\int_0^t\int_{\Omega_\e^f}f\cdot \partial_tU_\e\di{x}\di{\tau}+\int_0^t\int_{\Omega_\e^g}g\cdot\partial_tU_\e\di{x}\di{\tau}.
\end{multline*}
Further estimating the terms with time derivative gives us (here, $c_\mathcal{A}$ denotes the minimum of the positivity constants of $\mathcal{C}$ and $\mathcal{D}$)
\begin{multline*}
c\|p_\e(t)\|_{L^2(\Omega_\e^g)}^2
+c_{\mathcal{A}}\|e(U_\e)(t)\|^2_{L^2(\Omega)}
+2\e^2 c_K\int_0^t\|\nabla p_\e\|^2_{L^2(\Omega_\e^g)}\di{\tau}\\
\leq c\|p_\e(0)\|^2_{L^2(\Omega_\e^f)}
+c_{\mathcal{A}}\|e(U_\e)(0)\|^2_{L^2(\Omega)}\\
+2\int_0^t\int_{\Omega_\e^g}hp_\e\di{x}\di{\tau}
+2\int_0^t\int_{\Omega_\e^f}f\cdot \partial_tU_\e\di{x}\di{\tau}
+2\int_0^t\int_{\Omega_\e^g}g\cdot\partial_tU_\e\di{x}\di{\tau}.
\end{multline*}
Integration by parts with respect to time gives
\begin{align*}
\int_0^t\int_{\Omega_\e^f}f\cdot \partial_tU_\e\di{x}\di{\tau}
=-\int_0^t\int_{\Omega_\e^f}\partial_tf\cdot U_\e\di{x}\di{\tau}
+\left[\int_{\Omega_\e^f}f\cdot U_\e\di{x}\right]_0^t.
\end{align*}
As a consequence, we are led to (using also $p_\e(0)=0$ and $U_\e(0)=-E^{-1}_\e\mathcal{F}_\e(0)$)
\begin{multline*}
c\|p_\e(t)\|_{L^2(\Omega_\e^g)}^2
+c_{\mathcal{A}}\|e(U_\e)(t)\|^2_{L^2(\Omega)}
+2\e^2 c_K\int_0^t\|\nabla p_\e\|^2_{L^2(\Omega_\e^g)}\di{\tau}\\
\leq c_{\mathcal{A}}C_E\left(\|f(0)\|^2_{L^2(\Omega_\e^f)}+\|g(0)\|^2_{L^2(\Omega_\e^g)}\right)
+2\int_0^t\int_{\Omega_\e^g}hp_\e\di{x}\di{\tau}\\
-2\int_0^t\int_{\Omega_\e^f}\partial_tf\cdot U_\e\di{x}\di{\tau}
-2\int_0^t\int_{\Omega_\e^g}\partial_tg\cdot U_\e\di{x}\di{\tau}\\
+\left[\int_{\Omega_\e^f}f\cdot U_\e\di{x}+\int_{\Omega_\e^g}g\cdot U_\e\di{x}\right]_0^t.
\end{multline*}
Here, $C_E$ denotes the continuity constant of the operator $E_\e$.
With $c_{Ko}>0$ denoting the Korn's inequality constant, we then have 
\begin{multline*}
c\|p_\e(t)\|_{L^2(\Omega_\e^g)}^2
+c_{Ko}c_{\mathcal{A}}\|U_\e(t)\|^2_{H^1_0(\Omega)}
+2\e^2 c_K\int_0^t\|\nabla p_\e\|^2_{L^2(\Omega_\e^g)}\di{\tau}\\
\leq c_{\mathcal{A}}C_E\left(\|f(0)\|^2_{L^2(\Omega_\e^f)}+\|g(0)\|^2_{L^2(\Omega_\e^g)}\right)
+2\int_0^t\int_{\Omega_\e^g}hp_\e\di{x}\di{\tau}\\
-2\int_0^t\int_{\Omega_\e^f}\partial_tf\cdot U_\e\di{x}\di{\tau}
-2\int_0^t\int_{\Omega_\e^g}\partial_tg\cdot U_\e\di{x}\di{\tau}\\
+\left[\int_{\Omega_\e^f}f\cdot U_\e\di{x}+\int_{\Omega_\e^g}g\cdot U_\e\di{x}\right]_0^t.
\end{multline*}
Then, applying Young's inequality and setting $\tilde{c}=c_{Ko}c_{\mathcal{A}}$, we are led to
\begin{multline*}
c\|p_\e(t)\|_{L^2(\Omega_\e^g)}^2
+\frac{\tilde{c}}{2}\|U_\e(t)\|^2_{H^1_0(\Omega)}
+2\e^2 c_K\int_0^t\|\nabla p_\e\|^2_{L^2(\Omega_\e^g)}\di{\tau}\\
\leq\left(c_{\mathcal{A}}C_E+\frac{C_E}{2}+\frac{1}{2}\right)\left(\|f(0)\|^2_{L^2(\Omega_\e^f)}+\|g(0)\|^2_{L^2(\Omega_\e^g)}\right)
\\
+2\int_0^t\left(\|h\|_{L^2(\Omega_\e^g)}^2+\|\partial_tf\|_{L^2(\Omega_\e^f)}^2+\|\partial_tg\|_{L^2(\Omega_\e^g)}^2\right)\di{\tau}\\
+2\int_0^t\left(\|U_\e\|_{L^2(\Omega)}^2+\|p_\e\|_{L^2(\Omega_\e^g)}^2\right)\di{\tau}
+\frac{1}{2\tilde{c}}\left(\|f\|^2_{C(\overline{S};L^2(\Omega_\e^f))}+\|g\|^2_{C(\overline{S};L^2(\Omega_\e^g))}\right).
\end{multline*}
Finally, using Gronwall's inequality, we conclude that there is $C>0$, which is independent of the choice of $\e$, such that
\begin{multline*}
\|p_\e\|_{L^\infty(S;L^2(\Omega_\e^g))}^2
+\|U_\e(t)\|^2_{L^\infty(S;H^1_0(\Omega))}
+\e^2\|\nabla p_\e\|^2_{L^2(S\times\Omega_\e^g)}\\
\leq C\Bigg(\|f\|^2_{C^1(\overline{S};L^2(\Omega_\e^f)}+\|g\|^2_{C^1((\overline{S};L^2(\Omega_\e^g)}
+\|h\|_{L^2(S\times\Omega_\e^g)}^2
\Bigg).
\end{multline*}
\end{proof}

\section{Homogenization}\label{sec:4}
In this section, we are considering the limit process $\e\to0$ in the context of the two-scale convergence technique.
For the convenience of the reader, we shortly recall the definition and present the main results used here in the appendix.

%

We notice that the two-scale convergence is defined for the fixed domain and as the solution $u_\e$, $w_\e$ and $p_\e$ are defined on the domains $\Omega^g_\e$, $\Omega^f_\e$ and $\Omega^g_\e$, respectively, in order to apply the definition and the results of two-scale convergence, we need to define the solution on the whole domain $\Omega$.
Generally speaking, in a nonlinear setting this would require the use of so called \emph{extension operators} (see, e.g., \cite{AC92,M03}); since we are working with a linear problem simply extending by zero is sufficient.

In the following, for every function $\psi$ defined on either $\Omega_\e^f$ or $\Omega_\e^g$, $\widehat{\psi}$ will denote the zero extension the the whole of $\Omega$.
With that, we can discuss the two scale limit of $u_\e, w_\e, p_\e$ and their derivatives. This is addressed in the following lemma:

\begin{lemma}\label{lem:hom_limit}
There exist functions $U\in L^2(S;H^{1}(\Omega))^3$, $p\in L^2(S\times\Omega;H^{1}(Y))^3$, and $u^1\in L^2(S\times\Omega;H^{1}(Y))^3$ such that
\begin{alignat*}{2}
i)\ \widehat{U_\e}\to U ,\quad ii)\  \widehat{e(U_\e)}\overset{2}{\rightharpoonup}e(U)+e_y(u^1),\quad
iii)\  p_\e\overset{2}{\rightharpoonup}p ,\quad iv)\  \e\nabla p_\e\overset{2}{\rightharpoonup}\nabla_yp.
\end{alignat*}
\end{lemma}

\begin{proof}
The convergences $(i)$, $(iii)$, and $(iv)$ follow from the a-priori estimates given by \Cref{3:lem_estimates} and Lemmas \ref{lem:5.2}, \ref{lem:5.6} and \ref{lem:5.7}.
Moreover, we have $u^1\in L^2(S\times\Omega;H^{1}(Y))^3$ such that

$$
\nabla U_\e\overset{2}{\rightharpoonup}\nabla U+\nabla_yu^1.
$$
Let us choose $\phi\in C_0^\infty(S\times \Omega)^3$ and see that

\begin{align*}
\lim_{\e\to 0}\int_{S\times\Omega}e(u_\e)\phi\di{x}\di{t}&=\frac{1}{2}\lim_{\e\to 0}\int_{S\times\Omega}[\nabla u_\e+(\nabla u_\e)^t]\phi\di{x}\di{t}\\
&\overset{2}{\rightharpoonup}\frac{1}{2}\int_{S\times\Omega\times Y}\left[(\nabla u+\nabla_y u^1)\phi+(\nabla u+\nabla_y u^1)^T\phi\right]\di{(x,y)}\di{t}\\
&=\int_{S\times\Omega\times Y}\left(e(U)+e_y(u^1)\right)\phi\di{(x,y)}\di{t}.\notag
\end{align*}
\end{proof}

Our goal is to pass the two-scale limit in each equation of the model (\ref{eq:2.1a})-(\ref{seq:8}) using the limits given in \Cref{lem:hom_limit}.
%
%
%
%
We will first pass the two-scale limit in the momentum equation. 
To that end, let $\phi_0\in C_0^\infty(\Omega)^3$ and $\phi_1\in C_0^\infty(\Omega;C^\infty_{\#}(Y))^3$ such that we choose the test function as $v_\e(x)=v_0(x)+\e v_1(x,\frac{x}{\e})$ in \cref{weak.1}, i.e.,

\begin{equation*}
\int_\Omega\mathcal{A}\,e(U_\e)\colon e(v_\e)\di{x}
-\int_{\Omega}\alpha \widehat{p_\e}\nabla\cdot v_\e\di{x}\\
=\int_{\Omega}\left(\widehat{f}+\widehat{g}\right)\cdot v_\e\di{x}.
\end{equation*}
Now, due to $\nabla v_\e=\nabla v_0+\e \nabla_xv_1 +\nabla v_1$ (here, $x$ and $y$ denote the differentiation with respect to $x\in\Omega$ and $y\in Y$, respectively), this leads to

\begin{multline}\label{hom_1}
\int_\Omega\mathcal{A}_\e\,e(U_\e)\colon\left(e(v_0)+\e e_x(v_1) +e_y(v_1)\right)\di{x}\\
-\int_{\Omega}\alpha\,\widehat{p_\e}\left(\nabla\cdot v_0+\e \nabla_x\cdot v_1 +\nabla_y\cdot v_1\right)\di{x}
=\int_{\Omega}\left(\widehat{f}+\widehat{g}\right)\cdot\left(v_0+\e v_1\right)\di{x}.
\end{multline}
As the $L^2$ norms of $e(U_\e)$, $\widehat{p_\e}$, $\widehat{f}$, and $\widehat{g}$ are bounded with respect to $\e$, we find that the integrals

$$
\e\int_\Omega\mathcal{A}_\e\,e(U_\e)\colon e_x(v_1)\di{x},\quad\e\int_{\Omega}\alpha\,\widehat{p_\e}\nabla_x\cdot v_1\di{x},\quad\e\int_{\Omega}\left(\widehat{f}+\widehat{g}\right)\cdot v_1\di{x}
$$
converge to 0 for $\e\to0$.

Passing to the two-scale limits as $\e\to0$ in \cref{hom_1}, we are therefore led to

\begin{multline}\label{eq:hom_2}
\int_{\Omega\times Y}\mathcal{A}\left(e(U)+e_y(U^1)\right)\colon\left(e(v_0)+e_y(v_1)\right)\di{(x,y)}\\
-\int_{\Omega\times Y}\chi_g\alpha\,p\left(\nabla\cdot v_0+\nabla_y\cdot v_1\right)\di{(x,y)}
=\int_{\Omega\times Y}\left(\chi_ff+\chi_gg\right)\cdot v_0\di{(x,y)}
\end{multline}

Here, $\chi_f$ and $\chi_g$ denote the charateristic function of $Y_f$ and $Y_g$, respectively, and $\mathcal{A}=\chi_f\mathcal{C}+\chi_g\mathcal{D}$.
By density arguments, \cref{eq:hom_2} also holds true for all $(v_0,v_1)\in W_0^{1,2}(\Omega)^3\times L^2(\Omega;W^{1,2}_\#(Y))^3$.

Now, when choosing $v_0=0$, we arrive at the problem

$$
\int_{\Omega\times Y}\mathcal{A}\left(e(U)+e_y(U^1)\right)e_y(v_1)\di{(x,y)}\\
=\int_{\Omega\times Y}\chi_g\alpha\,p\nabla_y\cdot v_1\di{(x,y)}
$$
for all $v_1\in L^2(\Omega;W^{1,2}_\#(Y))^3$, which can be localized to

$$
\int_{Y}\mathcal{A}\left(e(U)+e_y(U^1)\right)e_y(v_1)\di{y}\\
=\int_{Y}\chi_g\alpha\,p\nabla_y\cdot v_1\di{y}\quad (v_1\in W^{1,2}_\#(Y)^3,\ \text{a.e.~in }\Omega).
$$
Similarly, choosing $\phi_1=0$, we get

\begin{multline}\label{eq:hom_3}
\int_{\Omega\times Y}\mathcal{A}\left(e(U)+e_y(U^1)\right)\colon e(v_0)\di{(x,y)}\\
-\int_{\Omega\times Y}\chi_g\alpha\,p\nabla\cdot v_0\di{(x,y)}
=\int_{\Omega\times Y}\left(\chi_ff+\chi_gg\right)\cdot v_0\di{(x,y)}
\end{multline}
for all $v_0\in W^{1,2}_0(\Omega)^3$.

Next, we will pass the two-scale limit in (\ref{eq:2.1c}).
We choose the test function $\phi_\e=\phi(t,x,x/\e)$ for $\phi\in C_0^\infty(S\times \Omega;C^\infty_{\#}(Y))$ and get

\begin{multline*}
-\int_{S}\int_\Omega\chi_{\Omega^g_\e}(c\widehat{p_\e}+\alpha\nabla \cdot U_\e)\partial_t\phi_\e\di{x}\di{t}+\e^2\int_{S}\int_\Omega\chi_{\Omega^g_\e}\mathcal{K}\widehat{\nabla p_\e}\cdot(\nabla \phi_\e+\e^{-1}\nabla_y \phi_\e)\di{x}\di{t}\\
=\int_{S}\int_\Omega\chi_{\Omega^g_\e}\widehat{h}\phi_\e\di{x}\di{t}
\end{multline*}

We pass to the two-scale limit as $\e\to 0$

\begin{multline*}
-\int_S\int_{\Omega\times Y}\chi_{g}\left(cp+\alpha(\nabla U+\nabla_y\cdot U)\right)\partial_t\phi\di{(x,y)}\di{t}\\
+\int_S\int_{\Omega\times Y}\chi_g\mathcal{K}\nabla_y p\cdot\nabla \phi\di{(x,y)}\di{t}
=\int_S\int_{\Omega\times Y}\chi_gh\phi\di{(x,y)}\di{t}.
\end{multline*}
which, by density, holds true for all $\phi\in L^2(S\times\Omega;W^{1,2}_\#(Y))$.
As all terms are restricted to $Y_g$, we can restrict to $\phi\in L^2(S\times\Omega;W^{1,2}(Y^g))$ (note that the periodicity property disappears since $Y_g$ does not touch $\Gamma$).
Moreover, we can localize in $x\in\Omega$ and arrive at 

\begin{multline*}
-\int_S\int_{Y^g}\left(cp+\alpha(\nabla U+\nabla_y\cdot U)\right)\partial_t\phi\di{y}\di{t}\\
+\int_S\int_{Y^g}\mathcal{K}\nabla_y p\cdot\nabla \phi\di{y}\di{t}
=\int_S\int_{Y^g}h\phi\di{y}\di{t}.
\end{multline*}
which  holds true for all $\phi\in L^2(S;W^{1,2}_\#(Y))$ almost eveywhere in $\Omega$.

Summarizing this limit process, we obtain the following system of variational equalities given via

\begin{subequations}
\begin{align}\label{hom:vara}
\int_{Y}\mathcal{A}\left(e(U)+e_y(U^1)\right)\colon e_y(v_1)\di{y}
=\int_{Y^g}\alpha\,p\nabla_y\cdot v_1\di{y},
\end{align}
\begin{multline}\label{hom:varb}
\int_{\Omega\times Y}\mathcal{A}\left(e(U)+e_y(U^1)\right)\colon e(v_0)\di{(x,y)}\\
-\int_{\Omega\times Y^g}\alpha\,p\nabla\cdot v_0\di{(x,y)}
=\int_{\Omega\times Y}\left(\chi_ff+\chi_gg\right)\cdot v_0\di{(x,y)},
\end{multline}
\begin{multline}\label{hom:varc}
-\int_S\int_{Y^g}\left(cp+\alpha(\nabla U+\nabla_y\cdot U)\right)\partial_t\phi\di{y}\di{t}\\
+\int_S\int_{Y^g}\mathcal{K}\nabla_y p\cdot\nabla \phi\di{y}\di{t}
=\int_S\int_{Y^g}h\phi\di{y}\di{t}
\end{multline}
for all $(v_0,v_1,\phi)\in W^{1,2}_0(\Omega)^3\times W^{1,2}_\#(Y)^3\times L^2(S;W^{1,2}(Y^g))$.
\end{subequations}

We go on by introducing cell problems and effective quantities in order to get a more accessible form of the homogenization limit.
For $j,k=\{1,2,3\}$, $t\in S$, and $x\in\Omega$ let $\tau_{jk}(t,x,\cdot),\vartheta(t,x,\cdot)\in W^{1,2}_\#(Y)^3$ be solutions of

\begin{subequations}
\begin{align}
0&=\int_{Y}\mathcal{A}\, e_y(\tau_{jk}+y_je_k)\colon e_y(v_1)\di{y}\label{sys:h1}.
\end{align}
for all $v_1\in W^{1,2}_0(\Omega)^3$.
Here, $e_k$ denotes the $k$-th unit vector and $y_j$ the $j$-th coordinate of $y\in Y$.

\begin{remark}
Solutions of the variational problems of \cref{sys:h1} are unique up to constants.
Utilizing Korns inequality, this can be shown via the Lemma of Lax-Milgram with respect to the Banach space of functions with zero average $\{u\in W^{1,2}_0(\Omega)^3\ : \ \int_Y u\di{y}=0\}$.
\end{remark}

We introduce the constant effective elasticity tensor $\mathcal{A}^h\in\R^{3\times3\times3\times3}$, the constant effective Biot-Willis parameter $\alpha^h\in\R^{3\times3}$, as well as the averaged volume force densities $F\colon S\times\Omega\to\R^3$ via
\begin{align}
(\mathcal{A}^h)_{ijkl}&=\int_Y\mathcal{A}\,e_y(\tau_{ij}+y_ie_j)\colon e_y(\tau_{kl}+y_ke_l)\di{y},\label{sys:h3}\\
\alpha^h&=\alpha\left(1+\int_{Y}\nabla_y\tau_{jk}\di{y}\right),\label{sys:h4}\\
F(t,x)&=\int_{Y^f}f\di{y}+\int_{Y^g}g\di{y}\label{sys:h5}.
\end{align}

\end{subequations}
We introduce the function

$$
\tilde{u}\colon S\times\Omega\times Y\to\R^3\quad\text{via}\quad \tilde{u}(t,x,y)=\sum_{j,k=1}^3\tau_{jk}(y)\left(e(U)(t,x)\right)_{jk}-U^1(t,x,y).
$$

\begin{remark}
In the standard mechnanical case, without the pressure, we would have $\tilde{u}=0$ (as $U^1$ can then be represented as a linear combination of derivatives of $U$ and the solution of the cell problem \cref{sys:h1}).
When $p$ is constant over $Y$, we have $\tilde{u}=\tau p$, where $\tau(t,x,\cdot)\in W^{1,2}_\#(Y)$ is a solution of the cell problem
\begin{align*}
0&=\int_{Y}\mathcal{A}\, e_y(\tau)\colon e_y(v_1)\di{y}-\int_{Y^g}\alpha\nabla\cdot v_1\di{y}.
\end{align*}
\end{remark}

Expressing $U^1$ in terms of $U$ and $\tilde{u}$ and inserting it into the variational equality \cref{hom:vara}, we calculate using the cell problem \cref{sys:h1} that

\begin{equation}
\int_{Y}\mathcal{A}\, e_y(\tilde{u})\colon e_y(v_1)\di{y}
=\int_{Y^g}\alpha\,p\nabla_y\cdot v_1\di{y}.\label{sys:h6}
\end{equation}
In its localized form, this corresponds to the PDE ($\tilde{u}=\widehat{\tilde{u}_f}+\widehat{\tilde{u}_g}$)

\begin{alignat*}{2}
-\nabla_y\cdot\left(\mathcal{C}\, e_y(\tilde{u}_f)\right)&=0&\quad&\text{in}\ \ Y^f,\\
-\nabla_y\cdot\left(\mathcal{D}\, e_y(\tilde{u}_g)-\alpha\,p\mathds{I}_3\right)&=0&\quad&\text{in}\ \ Y^g,\\
-\left(\mathcal{D}\, e_y(\tilde{u}_g)-\alpha\,p\mathds{I}_3\right)n_\Gamma&=-\mathcal{C}\, e_y(\tilde{u}_f)n_\Gamma&\quad&\text{on}\ \ \Gamma,\\
\tilde{u}_g&=\tilde{u}_f&\quad&\text{on}\ \ \Gamma,\\
y&\mapsto \tilde{u}\quad &&Y\text{-periodic}.
\end{alignat*}

With the effective elasticity tensor $A^h$ and the effective Biot-Willis matrix $\alpha^h$, the system given by equations \cref{hom:vara,hom:varb,hom:varc} for $(U,U^1,p)$ can equivalently be written as a problem for $(U,\tilde{u},p)$:

\begin{subequations}
\begin{align}\label{homalt:vara}
\int_{Y}\mathcal{A}\, e_y(\tilde{u})\colon e_y(v_1)\di{y}
=\int_{Y^g}\alpha\,p\nabla_y\cdot v_1\di{y},
\end{align}
\begin{multline}\label{homalt:varb}
\int_{\Omega}\mathcal{A}^h\, e(U)\colon e(v_0)\di{x}
+\int_{\Omega}\int_{Y}\mathcal{A}\,e_y(\tilde{u})\di{y}\colon e(v_0)\di{x}\\
-\int_{\Omega}\alpha\int_{Y_g}p\di{y}\nabla\cdot v_0\di{x}
=\int_{\Omega}F\cdot v_0\di{x},
\end{multline}
\begin{multline}\label{homalt:varc}
-\int_S\int_{Y^g}\left(cp+\alpha^h\colon U+\dive_y\tilde{u}\right)\partial_t\phi\di{y}\di{t}\\
+\int_S\int_{Y^g}\mathcal{K}\nabla_y p\cdot\nabla \phi\di{y}\di{t}
=\int_S\int_{Y^g}h\phi\di{y}\di{t}
\end{multline}
for all $(v_0,v_1,\phi)\in W^{1,2}_0(\Omega)^3\times W^{1,2}_\#(Y)^3\times L^2(S;W^{1,2}(Y^g))$.
\end{subequations}

\begin{remark}
In this form, the function $\tilde{u}$ can me interpreted as a micro deformation which leads to additional stresses in the macro mechanical problem.
Please note that the above system of three equations is strongly coupled.
Solving \cref{homalt:vara} for $\tilde{u}$ in terms of $p$ and introducing the corresponding linear and continuous solution operator $\mathcal{L}\colon L^2(Y^f)\to H^1(Y)$, we could substitute $\tilde{u}=\mathcal{L}p$ in \cref{homalt:varb,homalt:varc}, thereby eliminating the variable $\tilde{u}$.

\end{remark}

The system given by \cref{homalt:vara,homalt:varb,homalt:varc} corresponds to the following system of PDEs:

\begin{subequations}
\begin{mybox}{Limit problem in (strong) PDE form}
\begin{mybox}{Effective, macroscopic mechanics}
\vspace{-.3cm}
\begin{alignat}{2}
-\nabla\cdot\left(\mathcal{A}^he(U)+\mathcal{A}[e_y(\tilde{u})]_Y-\alpha[p]_{Y^g}\right)&=F&\quad&\text{in}\ \Omega,\label{hres1}\\
U&=0&\quad&\text{on}\ \partial\Omega.\label{hres2}
\end{alignat}
\end{mybox}

\begin{mybox}{Micro mechanical correction, $\tilde{u}=\widehat{\tilde{u}_f}+\widehat{\tilde{u}_g}$}
\vspace{-.3cm}
\begin{alignat}{2}
-\nabla_y\cdot\left(\mathcal{C}\, e_y(\tilde{u}_f)\right)&=0&\quad&\text{in}\ \ Y^f,\label{hres3}\\
-\nabla_y\cdot\left(\mathcal{D}\, e_y(\tilde{u}_g)-\alpha\,p\mathds{I}_3\right)&=0&\quad&\text{in}\ \ Y^g,\label{hres4}\\
-\left(\mathcal{D}\, e_y(\tilde{u}_g)-\alpha\,p\mathds{I}_3\right)n_\Gamma&=-\mathcal{C}\, e_y(\tilde{u}_f)n_\Gamma&\quad&\text{on}\ \ \Gamma,\label{hres5}\\
\tilde{u}_g&=\tilde{u}_f&\quad&\text{on}\ \ \Gamma,\label{hres6}\\
y&\mapsto \tilde{u}\quad &&Y\text{-periodic}.\label{hres7}
\end{alignat}
\end{mybox}

\begin{mybox}{Micro pore pressure}
\vspace{-.3cm}
\begin{alignat}{2}
\partial_t\left(cp+\alpha^h:\nabla U+\alpha\dive_y\tilde{u}\right)-\nabla\cdot\left(\mathcal{K}\nabla_y p\right)
&=h&\quad&\text{in}\ \ S\times Y^f\label{hres8}\\
-K_y\nabla p\cdot n_\Gamma&=0&\quad&\text{on}\ \ S\times\Gamma,\label{hres9}\\
p(0)&=0&\quad&\text{on}\ \ \Gamma.\label{hres0}
\end{alignat}
\end{mybox}
\end{mybox}
\end{subequations}

Here, $[\psi]_A$ denotes integration of a function $\psi$ over a domain $A$.

\section*{Acknowledgments}
The second author would like to thank the \emph{AG Modellierung und PDEs} at the University of Bremen for the kind invitation to visit their work group and work on  this problem.

\bibliographystyle{plain}
\bibliography{references}

\section{Appendix}

\subsection{\texttt{\textsf{Two-scale Convergence}}}\label{sec3.5}

Let $1\le p,q,<\infty$ such that $\frac{1}{p}+\frac{1}{q}=1$.

\begin{definition}\label{thm3.5.1}
Let $\varepsilon $ be a sequence of positive real numbers converging to 0. A sequence of functions $\left(u_\varepsilon\right)_{\varepsilon >0}$ in $L^{p}(\Omega)$ is said to two-scale convergent to a limit $u\in L^{p}(\Omega \times Y)$ if
\begin{equation}
\underset{\varepsilon \to 0}{\lim}\int_{\Omega}u_{\varepsilon}(x)\phi (x,\frac{x}{\varepsilon})\,dx=\int_{\Omega}\int_{Y}u(x,y)\phi (x,y)\,dx\,dy,\label{eqn3.5.1}
\end{equation}
\textnormal{ for all }$\phi \in L^{q}(\Omega; C_{\#}(Y))$.\footnote{$C_{\#}(Y)$ denotes the space of $Y$-periodic continuous functions in $y\in Y$.}
\end{definition}
If $(u_{\varepsilon})_{\varepsilon >0}$ is two-scale convergent to $u$ then we write $u_{\varepsilon}\overset{2}{\rightharpoonup}u$. The above definition is followed from the following theorem which is proved by Nguetseng (cf. Theorem 1 in \cite{Ngu89}):
\begin{lemma}\label{lem:5.2}
For every bounded sequence, $(u_{\varepsilon})_{\varepsilon >0}$, in $L^{p}(\Omega)$ there exist a subsequence $(u_{\varepsilon})_{\varepsilon >0}$ (still denoted by same symbol) and a $u\in L^{p}(\Omega \times Y)$ such that $u_{\varepsilon}\overset{2}{\rightharpoonup}u$.
\end{lemma}
In the defintion \ref{thm3.5.1}, one can notice that the space of test functions is chosen as $L^{q}(\Omega;C_{\#}(Y))$, but we can replace the space of test functions by $C_{0}^{\infty}(\Omega; C_{\#}^{\infty}(Y))$, if $(u_{\varepsilon})_{\varepsilon >0}$ satisfies certain condition which is given in the following theorem:

\begin{lemma}
Let $(u_{\varepsilon})_{\varepsilon >0}$ be bounded in $L^{p}(\Omega)$ such that
\begin{equation}
\underset{\varepsilon \to 0}{\lim}\int_{\Omega}u_{\varepsilon}(x)\phi(x,\frac{x}{\varepsilon})\,dx=\int_{\Omega}\int_{Y}u(x,y)\phi (x,y)\,dx\,dy \textnormal{ for all }\phi \in C_{0}^{\infty}(\Omega;C_{\#}^{\infty}(Y)). \label{eqn3.5.2}
\end{equation}
Then $(u_{\varepsilon})_{\varepsilon >0}$ is two-scale convergent to $u$.
\end{lemma}
We state some theorems on two-scale convergence. The proofs of all these theorems can be found in \cite{Ngu89}, \cite{LNW02}, \cite{All92}, \cite{DC99}.
\begin{lemma}\label{lem:5.4}
Let $(u_{\varepsilon})_{\varepsilon >0}$ be strongly convergent to $u \in L^{p}(\Omega)$, then $(u_{\varepsilon})_{\varepsilon >0}$ is two-scale convergent to $u_{1}(x,y)=u(x)$.
\end{lemma}
\begin{lemma}\label{lem:5.5}
Let $(u_{\varepsilon})_{\varepsilon >0}$ be two-scale convergent to $u$ in $L^{p}(\Omega \times Y)$, then $(u_{\varepsilon})_{\varepsilon >0}$ is weakly convergent to $\int_{Y}u(x,y)\,dy$ in $L^{p}(\Omega)$ and $(u_{\varepsilon})_{\varepsilon >0}$ is bounded.
\end{lemma}
\begin{lemma}\label{lem:5.6}
Let $\left(u_\varepsilon\right)_{\varepsilon >0}$ be a sequence in $H^{1,p}(\Omega)$ such that $u_{\varepsilon}\rightharpoonup u$ in $H^{1,p}(\Omega)$. Then $\left(u_\varepsilon\right)_{\varepsilon >0}$ two-scale converges to u and there exist a subsequence $\varepsilon$, still denoted by same symbol, and a $u_{1}\in L^{p}(\Omega; H^{1,p}_{\#}(Y))$ such that $\nabla _{x}u_{\varepsilon}\overset{2}{\rightharpoonup} \nabla u+\nabla _{y} u_{1}$.
\end{lemma}
\begin{lemma}\label{lem:5.7}
Let $(u_\varepsilon)_{\varepsilon >0}$ and $(\varepsilon \nabla _{x}u_{\varepsilon})_{\varepsilon >0}$ be bounded in $L^{p}(\Omega)$ and $[L^{p}(\Omega)]^{n}$ respectively. Then there exists $u\in L^{p}(\Omega;H^{1,p}_{per}(Y))$ such that up to a subsequence, still denoted by $\varepsilon$, we have
\[u_{\varepsilon}\overset{2}{\rightharpoonup} u \]
\[and \quad \varepsilon \nabla _{x}u_{\varepsilon}\overset{2}{\rightharpoonup} \nabla _{y}u\]
$as\quad \varepsilon \to 0$.
\end{lemma}
Since in this work we will only consider evolution equations which introduces time as an additional parameter, we generalize the definition \ref{thm3.5.1} to the functions depending on $t$ and $x$.
\begin{definition}
A sequence of functions $(u_{\varepsilon})_{\varepsilon >0}$ in $L^{p}((0,T)\times \Omega)$ is said to two-scale convergent to a limit $u\in L^{p}((0,T)\times\Omega \times Y)$ if 
\begin{equation}
\underset{\varepsilon \to 0}{lim}\int_{0}^{T}\int_{\Omega}u_{\varepsilon}(t,x)\phi (t,x,\frac{x}{\varepsilon})\di{x}\di{t}=\int_{0}^{T}\int_{\Omega}\int_{Y}u(t,x,y)\phi (t,x,y)\di{(x,y)}\di{t}\label{eqn3.5.3}
\end{equation}
\textnormal{ for all }$\phi \in L^{q}((0,T)\times \Omega; C_{per}(Y))$.
\end{definition}
All the above theorems on two-scale convergence can be generalized for the functions depending on $t$ and $x$.

\end{document}